\documentclass{gtart}

\newenvironment{To}
{ \begin{tikzpicture}[baseline=-0.5ex,scale=0.45,thick]
  \draw[clip] (0,0) circle (1); 
  \begin{scope}[thin] }
{ \end{scope}
  \end{tikzpicture}}

\newenvironment{TO}
{ \begin{tikzpicture}[baseline=-0.5ex,scale=0.60,thick]
  \draw[clip] (0,0) circle (1); 
  \begin{scope}[thin] }
{ \end{scope}
  \end{tikzpicture}}

\newcommand{\poscross} {%
     \draw (45:-1) -- (45:1);
     \draw[->] (45:-1) -- (45:0.5);
     \draw (135:-1) -- (135:-0.2);
     \draw (135:0.2) -- (135:1);
     \draw[->] (135:0.2) -- (135:0.5);
}

\newcommand{\negcross} {%
     \draw (135:-1) -- (135:1);
     \draw[->] (135:-1) -- (135:0.5);
     \draw (45:-1) -- (45:-0.2);
     \draw (45:0.2) -- (45:1);
     \draw[->] (45:0.2) -- (45:0.5);
}

\newcommand{\Tcup}[1] {%
     \draw[rotate=#1] (45:1)
        .. controls (45:0.7) and (0.5,0.5) .. (0.5,0) 
        .. controls (0.5,-0.5) and (315:0.7) .. (315:1);
}

\newcommand{\Ttick}[1] {%
     \draw[rotate=#1] (0:1) -- (0:0.5);
     \fill[rotate=#1] (0:0.5) circle (0.1);
}

\newcommand{\TTICK}[1] {%
     \draw[rotate=#1] (0:1) -- (0,0);
     \fill[rotate=#1] (0,0) circle (0.1);
}

\newcommand{\Tstrand}[1] { \draw[rotate=#1] (0:1) -- (180:1); }

\newcommand{\Tloop} { \draw (0,0) circle (0.5); }

\newcommand{\Ti} {%
     \draw (90:0.5) -- (270:0.5);
     \fill (90:0.5) circle (0.1);
     \fill (270:0.5) circle (0.1);
}

\newcommand{\Tdoto} { \fill (0,0) circle (0.1); }

\newcommand{\Tdot}[1] { \fill (#1:0.5) circle (0.1); }

\newcommand{\uparr}[1] {%
     \draw[->,rotate=#1]
        (315:1) .. controls (315:0.7) and (0.5,-0.5) .. (0.5,0);
}

\newcommand{\downarr}[1] {%
     \draw[->,rotate=#1]
        (45:1) .. controls (45:0.7) and (0.5,0.5) .. (0.5,0);
}

\newcommand{\Rone}[2] {%
   \draw[xscale=#1] (225:1) -- (225:0.2);
   \draw[->,xscale=#1] (135:0.2) -- (135:0.5);
   \draw[xscale=#1] (135:0.2) -- (135:1);
   \draw[xscale=#1] (315:0.2)
      .. controls (315:0.7) and (0.7,-0.2) .. (0.7,0)
      .. controls (0.7,0.2) and (45:0.7) .. (45:0.2);
   \draw[yscale=#2] (135:0.2) -- (135:-0.2);
}

\newcommand{\Rtwo} {%
    \draw[->] (225:1)
        .. controls (-0.2,-0.6) .. (0,-0.5)
        .. controls (0.2,-0.4) and (0.5,-0.2) .. (0.5,0);
    \draw (225:1)
        .. controls (-0.2,-0.6) .. (0,-0.5)
        .. controls (0.2,-0.4) and (0.5,-0.2) .. (0.5,0)
        .. controls (0.5,0.2) and (0.2,0.4) .. (0,0.5)
        .. controls (-0.2,0.6) .. (135:1);
    \draw (315:1)
        .. controls (0.2,-0.6) .. (0.15,-0.575);
    \draw[->] (-0.15,0.425)
        .. controls (-0.2,0.4) and (-0.5,0.2) .. (-0.5,0);
    \draw (-0.15,-0.425)
        .. controls (-0.2,-0.4) and (-0.5,-0.2) .. (-0.5,0)
        .. controls (-0.5,0.2) and (-0.2,0.4) .. (-0.15,0.425);
    \draw (0.15,0.575)
        .. controls (0.2,0.6) .. (45:1);
}

\newcommand{\Rthree}[1] {%
    \draw[xscale=#1] (270:1) arc (225:135:1.414);
    \draw[->,xscale=#1] (270:1) arc (225:180:1.414);
    \draw[->,xscale=#1,rotate around={28:(120:1)}]
         (210:1) arc (255:272:1.414);
    \draw[xscale=#1,rotate around={28:(120:1)}]
         (210:1) arc (255:289:1.414);
    \draw[->,xscale=#1,rotate=120,rotate around={28:(120:1)}]
         (210:1) arc (255:272:1.414);
    \draw[xscale=#1,rotate=120,rotate around={28:(120:1)}]
         (210:1) arc (255:289:1.414);
    \draw[xscale=#1] (210:1) arc (255:273:1.414);
    \draw[xscale=#1,rotate=120] (210:1) arc (255:273:1.414);
    \draw[xscale=#1,yscale=-1] (210:1) arc (255:273:1.414);
    \draw[xscale=#1,yscale=-1,rotate=120] (210:1) arc (255:273:1.414);
    \draw[rotate=90*#1] (300:1) arc (255:210:1.414);
}

\newcommand{\innerproduct} {%
  \begin{tikzpicture}[baseline=-1.0ex,scale=0.30]
  {
    \draw[thick] circle (4);
    \draw[thin] (-2,0.8) -- (2,0.8);
    \draw[thin] (-2,-0.8) -- (2,-0.8);
    \draw[thick, fill=white] (-2,0) node {$X$} circle(1);
    \draw (-3,0) node [left = -4] {\tiny$*$};
    \draw[thick, fill=white] (2,0) node [xshift=1pt] {$Y^*$} circle(1);
    \draw (3,0) node [right = -4] {\tiny$*$};
    \draw (0,0) node [above = -8.5] {$\vdots$};
  }
  \end{tikzpicture}
}
\usepackage{ifpdf}

\ifpdf
\usepackage[pdftex,plainpages=false,hypertexnames=false,pdfpagelabels]{hyperref}
\else
\usepackage[dvips,plainpages=false,hypertexnames=false]{hyperref}
\fi

\usepackage{amsmath,amssymb,amsfonts}
\usepackage{enumerate}

\usepackage{leftidx,soul}

\ifpdf
\usepackage[pdftex,all]{xy}
\else
\usepackage[all]{xy}
\fi

\SelectTips{cm}{}

\ifpdf

\else

\fi

\newcommand{\arxiv}[1]{\href{http://arxiv.org/abs/#1}{\tt arXiv:\nolinkurl{#1}}}

\newcommand{\googlebooks}[1]{(preview at \href{http://books.google.com/books?id=#1}{google books})}

\def\RCS$#1: #2 ${\expandafter\def\csname RCS#1\endcsname{#2}}
\RCS$Revision: 993 $ \RCS$Date: 2009-09-14 08:32:46 -0700 (Mon, 14 Sep 2009) $

\theoremstyle{plain}
\newtheorem{prop}{Proposition}[section]

\newtheorem{thm}[prop]{Theorem}
\newtheorem{lem}[prop]{Lemma}

\newtheorem*{cor*}{Corollary}
\newtheorem{defn}[prop]{Definition}         
\newtheorem*{defn*}{Definition}             
\newtheorem*{notation*}{Notation}

\numberwithin{equation}{section}

\def\clap#1{\hbox to 0pt{\hss#1\hss}}

\renewcommand{\imath}{\mathfrak{i}}
\renewcommand{\jmath}{\mathfrak{j}}

\ifpdf
\usepackage[pdftex]{graphicx}
\else
\usepackage[dvips]{graphicx}
\fi

\usepackage{tikz}
\usetikzlibrary{shapes}
\usetikzlibrary{backgrounds}
\usetikzlibrary{decorations,decorations.pathreplacing}

\usepackage{microtype}

\title{A diagrammatic Alexander invariant of tangles}

\author{Stephen~Bigelow}
\address{
}%

 \secondaryclass{
} 

\begin{document}

\begin{abstract}
We give a new construction of
the one-variable Alexander polynomial of an oriented knot or link,
and show that it generalizes
to a vector valued invariant of oriented tangles.
\end{abstract}

\maketitle

\section{Introduction}
The Alexander polynomial
is the unique invariant of oriented knots and tangles
that is one for the unknot
and satisfies the {\em Alexander-Conway skein relation}.
$$\begin{TO} \poscross \end{TO} - \begin{TO} \negcross \end{TO}
 = (q-q^{-1})\begin{TO}\Tcup{0}\Tcup{180} \downarr{180} \uparr{0} \end{TO}.$$
Many other equivalent definitions are known.
The aim of this paper is to give yet another definition of the
Alexander polynomial, which we will prove is equivalent to the above
skein theoretic definition.

An advantage of our definition
is that it generalizes immediately
to give an invariant of oriented tangles.
Other generalizations of the Alexander polynomial to tangles
have been given in \cite{cimtur} and \cite{arch}.
Their definitions are for the multivariable Alexander polynomial,
whereas this paper only concerns the single variable version.

Let $T$ be an oriented tangle diagram in a disk,
having two endpoints on the boundary of the disk.
We allow $T$ to contain more than one component:
one strand with both endpoints on the boundary of the disk,
and possibly other strands that form closed loops.
Let $\hat{T}$ denote the closure of $T$,
that is,
the oriented knot or link obtained by connecting the two endpoints of $T$.
Our construction of the Alexander polynomial of $\hat{T}$
is best described as an invariant of $T$.

In Section \ref{sec:defn}, we will define the invariant $\Delta(T)$.
The definition is reminiscent of the Kauffman bracket \cite{kauffman},
in that it is a state sum over a certain kind of resolutions of the crossings.
In Sections \ref{sec:pa} and \ref{sec:pa2},
we use planar algebras to study
the relevant formal linear combinations of diagrams.
In Section \ref{sec:proof},
we use our findings to prove that
$\Delta(T)$ is the Alexander polynomial of $\hat{T}$,
after multiplication by an appropriate monomial $\pm q^k$.
If $T$ is a tangle with more than two endpoints
then we still obtain an invariant $\Delta(T)$,
which is a linear combination of a finite number of simple diagrams.

The idea for this paper began when I was visiting Vincent Florens
at the Universit\'e de Pau et des Pays de l'Adour,
and I would like to thank him for his kind hospitality,
help, and motivation for this work.

I would also like to thank Dror Bar-Natan and his students for their interest,
many helpful observations, and programming skill.
Bar-Natan used Mathematica to check the tedious hand calculations used in this paper,
and show that $\Delta(T)$ is invariant under
Naik and Stanford's doubled-delta move \cite{naikstanford}.
He also observed a parallel between my invariant
and the invariant defined by Archibald in \cite{arch}.
It seems almost certain that these invariants are in fact equivalent.

\section{Definition of the invariant}
\label{sec:defn}

Let $T$ be an oriented tangle diagram in a disk,
having two endpoints on the boundary of the disk.
In this section,
we define $\Delta(T)$.
Our definition is based on the following.
$$
\begin{To} \poscross \end{To}
= q \, \begin{To} \Tcup{0} \Tcup{180} \end{To}
+ q \left(
\begin{To} \Tstrand{135} \Ttick{45} \Ttick{225} \end{To}
- \begin{To} \Tcup{0} \Ttick{135} \Ttick{225} \end{To}
- \begin{To} \Ttick{45} \Ttick{135} \Ttick{225} \Ttick{315} \end{To}
\right)
+ q^{-1} \left(
\begin{To} \Tstrand{45} \Ttick{135} \Ttick{315} \end{To}
- \begin{To} \Tcup{180} \Ttick{45} \Ttick{315} \end{To}
- \begin{To} \Ttick{45} \Ttick{135} \Ttick{225} \Ttick{315} \end{To}
\right),
$$
$$
\begin{To} \negcross \end{To}
= q^{-1} \begin{To} \Tcup{0} \Tcup{180} \end{To}
+ q \left(
\begin{To} \Tstrand{135} \Ttick{45} \Ttick{225} \end{To}
- \begin{To} \Tcup{0} \Ttick{135} \Ttick{225} \end{To}
- \begin{To} \Ttick{45} \Ttick{135} \Ttick{225} \Ttick{315} \end{To}
\right)
+ q^{-1} \left(
\begin{To} \Tstrand{45} \Ttick{135} \Ttick{315} \end{To}
- \begin{To} \Tcup{180} \Ttick{45} \Ttick{315} \end{To}
- \begin{To} \Ttick{45} \Ttick{135} \Ttick{225} \Ttick{315} \end{To}
\right).
$$
Here,
a right- or left-handed crossing
is written as a formal linear combination of seven diagrams.
The coefficients are $\pm q^{\pm 1}$,
where $q$ can be taken to be a formal variable.
We allow strands to have endpoints in the interior of the diagram.

Apply the above rule in a multilinear fashion to all of the crossings in $T$.
If $T$ has $n$ crossings
then we obtain a sum of $7^n$ terms
$$T = \sum_{i=1}^{7^n} \lambda_i D_i,$$
where
each coefficient $\lambda_i$ is of the form $\pm q^{k_i}$,
and each $D_i$ is a diagram with no crossings.
We can forget the orientations on strands in $D_i$.

We will define $\Delta(T)$ to be a sum of some of the coefficients $\lambda_i$,
where the diagrams $D_i$ determine which coefficients to include in the sum.
Each $D_i$ is a disjoint union of embedded loops and edges,
where an edge may have zero, one, or both endpoints on the boundary of the disk.
Eliminate any $\lambda_i$ for which $D_i$ contains a loop,
or contains an edge with exactly one endpoint on the boundary of the disk,
and let $\Delta(T)$ be the sum of the remaining coefficients.
Thus $\Delta(T)$ is the sum of $\lambda_i$ taken over all $i$
such that $D_i$ has no closed loops,
one strand with both endpoints on the boundary of the disk,
and possibly some strands with both endpoints in the interior of the disk.

\section{A planar algebra}
\label{sec:pa}

Our definition of $\Delta(T)$ actually describes
a morphism from the planar algebra of oriented tangles
to a planar algebra $\mathcal{P}$ of unoriented $1$-valent graphs.
We can define $\mathcal{P}$ by generators and relations as follows.

\begin{defn}
Let $\mathcal{P}$ be the planar algebra given by the one generator:
$$\begin{To} \TTICK{270} \end{To}$$
and the two relations:
$$
\begin{To} \Tloop \end{To} = 0,
\quad \mbox{and}\quad 
\begin{To} \Ti \end{To} = \begin{To} \end{To}.
$$
\end{defn}

The aim of this section is to flesh out this definition
and give some basic properties of $\mathcal{P}$.

\begin{defn}
A {\em basis diagram} is
a collection of disjoint embedded edges in the disk,
each having either one or both endpoints on the boundary of the disk.
Every diagram also includes a basepoint on the boundary of the disk,
which may not coincide with the endpoint of any strand.
Two diagrams are considered the same if they are isotopic.
\end{defn}

Let $\mathcal{P}_n$ be the complex vector space
of formal linear combinations of basis diagrams
that have a total of $n$ endpoints on the boundary of the disk
(and possibly some endpoints in the interior of the disk).

A more general diagram in $\mathcal{P}_n$
may include closed loops,
or edges with both endpoints in the interior of the disk.
The defining relations state that
any diagram with a closed loop is zero,
and strands with both endpoints in the interior can be deleted.
I like to think of these loops and interior edges
as ``bubbles'' and ``confetti''.

The vector spaces $\mathcal{P}_n$ form a planar algebra $\mathcal{P}$.
We will not give a formal definition of a planar algebra here.
It should suffice to think of a planar algebra as
a collection of vector spaces of formal linear combinations of diagrams,
which can be connected together in arbitrary planar ways.
For a more detailed definition, see Jones \cite{jonesplanar}
(but note that our planar algebra $\mathcal{P}$ is not shaded,
and diagrams in $\mathcal{P}_n$ have $n$ endpoints as opposed to $2n$).

The vector space $\mathcal{P}_0$ is spanned by the empty diagram.
Let the {\em partition function}
$$Z \co \mathcal{P}_0 \to \mathbf{C}$$
be the isomorphism that takes the empty diagram to one.
We remark that $\mathcal{P}$ is {\em spherical}
in the sense that the partition function gives
a well-defined invariant of diagrams drawn on a sphere.

We define an {\em adjoint operation} $*$ on $\mathcal{P}$ as follows.
If $D$ is a diagram then $D^*$ is the mirror image of $D$.
Extend this to a conjugate-linear operation on $\mathcal{P}_n$ for all $n$.
We define an inner product on each space $\mathcal{P}_n$
to be the following sesquilinear operation.
$$ \langle X,Y \rangle = Z \left( \innerproduct \right).$$
The two small stars in the above diagram
indicate the basepoints on the boundary of $X$ and $Y^*$.
For the rest of this paper,
the basepoint can be taken to be at the far left of every diagram,
and will therefore be omitted.

We now introduce notation for an important element of $\mathcal{P}_2$.

\begin{defn}
Let a {\em dotted strand}
denote the following element of $\mathcal{P}_2$.
$$
\begin{To} \Tstrand{90} \Tdoto \end{To} =
\begin{To} \Tstrand{90} \end{To} - \begin{To} \Ttick{90} \Ttick{270} \end{To}.
$$
\end{defn}

Thus a diagram with $n$ dots on strands
is shorthand for a linear combination of $2^n$ diagrams.

\begin{lem}
\label{lem:dot}
The following relations hold in $\mathcal{P}$.
\begin{itemize}
\item $\begin{To} \TTICK{270} \Tdot{270} \end{To} = 0$,
\item $\begin{To} \Tstrand{90} 
       \fill (90:0.3) circle (0.1cm); 
       \fill (270:0.3) circle (0.1cm); 
       \end{To}
      = \begin{To} \Tstrand{90} \Tdoto \end{To}$,
\item $\begin{To} \Tloop \Tdot{0} \end{To} = - \, \begin{To} \end{To}$.
\end{itemize}
\end{lem}

\begin{proof} These are immediate from the definition and the defining
relations of $\mathcal{P}$.
\end{proof}

\begin{defn}
A {\em dotted basis diagram}
is a diagram in which
every strand is either a dotted strand
with both endpoints on the boundary of the disk,
or a non-dotted strand
with exactly one endpoint on the boundary of the disk.
\end{defn}

\begin{lem}
The dotted basis diagrams in $\mathcal{P}_n$
form a basis for $\mathcal{P}_n$.
\end{lem}

\begin{proof}
Rearranging the definition of a dotted strand,
we have
$$
\begin{To} \Tstrand{90} \end{To} =
\begin{To} \Tstrand{90} \Tdoto \end{To} +
\begin{To} \Ttick{90} \Ttick{270} \end{To}.
$$
We can use this to ensure that every strand is either dotted
or has at least one endpoint in the interior of the disk.
Now eliminate any closed loops
and strands that have both endpoints in the interior of the disk.
This shows that the dotted basis diagrams span $\mathcal{P}_n$.
The easiest way to see that they are linearly independent
is by a dimension count:
the number of dotted basis diagrams for $\mathcal{P}_n$ is the same as
the number of basis diagrams for $\mathcal{P}_n$.
\end{proof}

\section{An improved planar algebra}
\label{sec:pa2}

It turns out that $\mathcal{P}$ is
not exactly the best planar algebra to work with.
In this section,
we impose an additional relation
to obtain a new planar algebra $\mathcal{P}'$.
This is motivated by the following.

\begin{lem}
\label{lem:negligible}
If
$X = \begin{To} \Tcup{90} \Tdot{90} \Tcup{270} \Tdot{270} \end{To}
    + \begin{To} \Tcup{0} \Tdot{0} \Tcup{180} \Tdot{180} \end{To}$
then $\langle X,Y \rangle = 0$
for all $Y \in \mathcal{P}_4$.
\end{lem}

\begin{proof}
If any endpoint of $X$ leads to a univalent vertex
then both terms in $X$ become zero.
If any neighboring endpoints of $X$ are joined by a strand
then the two terms in $X$ cancel out.
One or both of these must happen in the computation of
$\langle X,Y \rangle$ for any diagram $Y$.
\end{proof}

The above proposition can be phrased as saying that
$X$ is {\em negligible} in $\mathcal{P}$.
It is common practice
to quotient out negligible elements of a planar algebra.

\begin{defn}
Let $\mathcal{P}'$ be the planar algebra given by the one generator:
$$\begin{To} \TTICK{270} \end{To}$$
and the three relations:
$$
\begin{To} \Tloop \end{To} = 0,
\quad 
\begin{To} \Ti \end{To} = \begin{To} \end{To},
\quad
\mbox{and}
\quad
\begin{To} \Tcup{90} \Tdot{90} \Tcup{270} \Tdot{270} \end{To}
  + \begin{To} \Tcup{0} \Tdot{0} \Tcup{180} \Tdot{180} \end{To} = 0.
$$
Call the third relation the {\em saddle relation}.
\end{defn}

\begin{lem}
$\mathcal{P}'_0$ is one-dimensional.
\end{lem}

\begin{proof}
It is a general fact
about spherical planar algebras $\mathcal{P}$
that taking the quotient by a negligible element $X$
has no effect on the space of closed diagrams $\mathcal{P}_0$.
This is because $\mathcal{P}'_0$ is the quotient of
$\mathcal{P}_0$ by the span of all elements
obtained by placing $X$ inside a larger diagram,
but such elements are already zero in $\mathcal{P}_0$
by Lemma \ref{lem:negligible}.
\end{proof}

We can define a partition function and inner product on $\mathcal{P}'$
as we did for $\mathcal{P}$.
The dotted basis diagrams still span $\mathcal{P}'_n$,
but they are no longer linearly independent.
We define an equivalence relation as follows.

\begin{defn}
Two dotted basis diagrams $D$ and $D'$ are {\em equivalent}
if they have the same number of dotted strands,
and the same set of endpoints of dotted strands.
\end{defn}

\begin{lem}
Suppose $D$ and $D'$ are dotted basis diagrams in $\mathcal{P}'_n$.
Then $\langle D,D' \rangle$ is $\pm 1$ if $D$ and $D'$ are equivalent,
and $0$ if they are not.
\end{lem}

\begin{proof}
Consider the closed diagram obtained
by connecting the corresponding endpoints of $D$ and $(D')^*$.
If $D$ and $D'$ are equivalent
then their dotted strands connect
to form some number of closed dotted loops,
and the undotted strands are connected
to form strands with both endpoints in the interior.
If $D$ and $D'$ are not equivalent
then a dotted strand from one of the diagrams
is connected to a strand with an interior endpoint from the other.
The result now follows from Lemma \ref{lem:dot}.
\end{proof}

\begin{lem}
Let $\mathcal{B}$ be a set consisting of one dotted basis diagram
from each equivalence class in $\mathcal{P}'_n$.
Then $\mathcal{B}$ is a basis for $\mathcal{P}'_n$.
\end{lem}

\begin{proof}
If $D$ and $D'$ are equivalent dotted basis diagrams in $\mathcal{P}'_n$
then,
by repeated application of the third defining relation of $\mathcal{P}'$,
we have $D = \pm D'$.
It follows that $\mathcal{B}$ spans $\mathcal{P}'_n$.
By the previous lemma,
$\mathcal{B}$ is orthogonal,
and hence linearly independent.
\end{proof}

We could use some convention
to precisely specify a basis for $\mathcal{P}'_n$,
although perhaps it is more elegant not to do so.
The dimension of $\mathcal{P}'_n$ is $2^{n-1}$,
the number of even subsets of the set of $n$ endpoints on the boundary.

\section{The main results}
\label{sec:proof}

Recall the definitions of a right- and left-handed crossing
from Section \ref{sec:defn}.
They are equivalent to the following expressions using dotted strands.
$$
\begin{To} \poscross \end{To}
= q \, \begin{To} \Tcup{0} \Tcup{180} \Tdot{180} \Tdot{0} \end{To}
+ (q-q^{-1}) \begin{To} \Tcup{180} \Ttick{45} \Ttick{315} \Tdot{180} \end{To}
+ q \, \begin{To} \Tstrand{135} \Ttick{45} \Ttick{225} \Tdoto \end{To}
+ q^{-1} \begin{To} \Tstrand{45} \Ttick{135} \Ttick{315} \Tdoto \end{To}
- q^{-1} \begin{To} \Ttick{45} \Ttick{135} \Ttick{225} \Ttick{315} \end{To}.
$$
$$
\begin{To} \negcross \end{To}
= q^{-1} \begin{To} \Tcup{0} \Tcup{180} \Tdot{180} \Tdot{0} \end{To}
+ (q^{-1}-q) \begin{To} \Tcup{0} \Ttick{135} \Ttick{225} \Tdot{0} \end{To}
+ q^{-1} \begin{To} \Tstrand{45} \Ttick{135} \Ttick{315} \Tdoto \end{To}
+ q \, \begin{To} \Tstrand{135} \Ttick{45} \Ttick{225} \Tdoto \end{To}
- q \, \begin{To} \Ttick{45} \Ttick{135} \Ttick{225} \Ttick{315} \end{To}.
$$

\begin{lem}
$\mathcal{P}$, and hence $\mathcal{P}'$,
satisfies the Alexander-Conway skein relation
$$\begin{To} \poscross \end{To} - \begin{To} \negcross \end{To}
= (q-q^{-1}) \begin{To} \Tcup{0}\Tcup{180}\downarr{180}\uparr{0} \end{To},$$
and the following variations on Reidemeister I.
$$\begin{To} \Rone{1}{1} \end{To}
= \begin{To} \Rone{1}{-1} \end{To}
= -q^{-1} \begin{To} \downarr{180} \Tcup{180} \end{To},$$
$$\begin{To} \Rone{-1}{1} \end{To}
= \begin{To} \Rone{-1}{-1} \end{To}
= -q \, \begin{To} \uparr{0} \Tcup{0} \end{To}.$$
\end{lem}

\begin{proof}
These follow easily from the definitions of crossings
and Lemma \ref{lem:dot}.
\end{proof}

\begin{lem}
$\mathcal{P}'$ satisfies the the following version of Reidemeister II.
$$\begin{TO} \Rtwo \end{TO} =
 \begin{TO} \downarr{0} \downarr{180} \Tcup{0} \Tcup{180} \end{TO}.$$
\end{lem}

\begin{proof}
First expand out each crossing in the diagram
on the left side of the equation
into a linear combination of five diagrams with dotted strands.
Most of the resulting $25$ diagrams can be eliminated by Lemma \ref{lem:dot}.
Now express the diagram on the right hand side of the equation
as a sum of four of the dotted basis diagrams of $\mathcal{P}_4$.
Combining these calculations gives
$$\begin{To} \Rtwo \end{To} -
 \begin{To} \downarr{0} \downarr{180} \Tcup{0} \Tcup{180} \end{To}
 = \begin{To} \Tcup{90} \Tdot{90} \Tcup{270} \Tdot{270} \end{To}
    + \begin{To} \Tcup{0} \Tdot{0} \Tcup{180} \Tdot{180} \end{To}.
$$
Thus the desired relation is equivalent to
the saddle relation in the definition of $\mathcal{P}'$.
\end{proof}

\begin{lem}
$\mathcal{P}$, and hence $\mathcal{P}'$,
satisfies the the following version of Reidemeister III.
$$\begin{TO} \Rthree{1} \end{TO} = \begin{TO} \Rthree{-1} \end{TO}.$$
\end{lem}

\begin{proof}
Each side of the equation can be expanded out to
a linear combination of $125$ terms.
However it is more efficient to expand one crossing at a time,
using Lemma \ref{lem:dot} to eliminate many terms as they arise.
An unpleasant number of diagrams remain,
but the resulting identity can be checked by hand
with some care and patience,
using only Lemma \ref{lem:dot}.
\end{proof}

Since we are working with oriented tangles,
there are other versions of Reidemeister II and III to check.

\begin{lem}
$\mathcal{P}'$ satisfies all versions of Reidemeister II and III.
\end{lem}

\begin{proof}
These can all be deduced from the relations we have already proved.
For example, consider what happens if
we connect the left two endpoints of the above Reidemeister III relation.
Using Reidemeister I and the above Reidemeister II,
we then obtain a new version of Reidemeister II.
We can also use the skein relation
to effectively reverse both of the crossings
in a Reidemeister II relation.
Finally,
in the presence of all versions of Reidemeister II,
all versions of Reidemeister III become equivalent.
\end{proof}

To make Reidemeister I hold precisely,
we use a correction factor based on the turning number.
The usual definition of the turning number of a curve
is the winding number of the tangent vector,
but the following pictorial definition is more in keeping
with the spirit of this paper.

Suppose $T$ is an oriented tangle with two endpoints.
Smooth every crossing of $T$ in the usual way:
$$\begin{To} \poscross \end{To} \mapsto
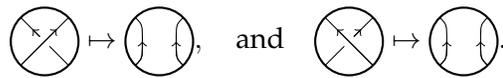
\begin{To} \Tcup{0}\Tcup{180}\downarr{180}\uparr{0} \end{To},
\quad \mbox{and} \quad
\begin{To} \negcross \end{To} \mapsto
\begin{To} \Tcup{0}\Tcup{180}\downarr{180}\uparr{0} \end{To}.$$
The resulting diagram consists of
a strand with both endpoints on the boundary
and some oriented closed loops.
Let the {\em turning number} $\tau(T)$ be
the number of positively oriented loops
minus the number of negatively oriented loops.

\begin{thm}
If $T$ is an oriented tangle with two endpoints on the boundary of the disk
then
$$(-q)^{-\tau(T)} \Delta(T)$$
is the Alexander polynomial of $\hat{T}$.
\end{thm}

\begin{proof}
Consider $T$ as an element of $\mathcal{P}'_2$.
By the skein relation and Reidemeister moves,
$$T = \lambda \, \begin{To} \Tstrand{90} \end{To},$$
for some scalar $\lambda$.
By our complete description of
$\mathcal{P}'_2$ and $\mathcal{P}_2$,
we know that they are isomorphic,
so the above equation holds in $\mathcal{P}_2$ as well.
By definition,
$\Delta(T) = \lambda$.

Both $\Delta(T)$ and $\tau(T)$ are invariant under Reidemeister II and III.
However they both change under the different versions of Reidemeister I.
The correction term was chosen precisely to ensure that
$$(-q)^{-\tau(T)} \Delta(T)$$
is invariant under all Reidemeister moves.

Note that $\Delta(T)$ satisfies the Alexander-Conway skein relation,
and all terms in this relation have the same turning number.
Thus $(-q)^{-\tau(T)} \Delta(T)$ also satisfies this skein relation.
Finally,
if $T$ is a single straight strand,
so that $\hat{T}$ is the unknot,
then $\Delta(T) = 1$ and then $\tau(T) = 0$.

We conclude that $(-q)^{-\tau(T)} \Delta(T)$ 
satisfies the definition of the Alexander polynomial of $\hat{T}$
as given in the introduction.
\end{proof}

This completes our construction of the Alexander polynomial.
As promised,
it generalizes to arbitrary oriented tangles.

\begin{defn}
If $T$ is an oriented tangle with $2n$ endpoints on the boundary of the disk
then let $\Delta(T)$ be the image of $T$ in $\mathcal{P}'_{2n}$,
using the definitions of crossings given in Section \ref{sec:defn}.
\end{defn}

By the discussion in Section \ref{sec:pa2},
it seems reasonable to say the vector space $\mathcal{P}'_{2n}$
is completely understood.
In particular, it has a basis,
which is canonical up to the sign of each basis vector,
and there is an elementary algorithm
to express any vector in terms of the basis vectors.

Thus we have a vector valued invariant $\Delta$ of oriented tangles
that satisfies the Alexander-Conway skein relation,
is invariant under Reidemeister II and III,
and is ``almost'' invariant under Reidemeister I.

It is possible to renormalize $\Delta(T)$
to fix the problem with Reidemeister I,
but this requires an arbitrary choice of convention
to specify the turning number of a tangle.
I prefer to avoid choosing conventions,
and leave $\Delta(T)$ as an invariant of
oriented tangle diagrams up to regular isotopy,
lying in a vector space that has a canonical basis vectors up to sign.


\end{document}